\documentclass[11pt,reqno]{amsart}
\usepackage[letterpaper,margin=1in,footskip=0.25in]{geometry}
\usepackage{mathrsfs}
\usepackage{microtype}
\usepackage{mathtools}
\usepackage{enumitem}
\PassOptionsToPackage{dvipsnames}{xcolor}
\usepackage{tikz-cd}
\PassOptionsToPackage{pdfusetitle,pagebackref,colorlinks}{hyperref}
\usepackage{bookmark}
\hypersetup{citecolor=OliveGreen,linkcolor=Mahogany,urlcolor=Plum}
\usepackage[hyphenbreaks]{breakurl}
\usepackage{amssymb}
\usepackage{amsfonts}
\usepackage{bbm}
\usepackage{enumitem}
\usepackage{upgreek}
\usepackage{bm}
\usepackage{stmaryrd}
\usepackage{mathrsfs}
\usepackage{comment}
\usepackage{graphicx}
\usepackage{todonotes}

\allowdisplaybreaks

\makeatletter
\providecommand{\leftsquigarrow}{%
	\mathrel{\mathpalette\reflect@squig\relax}%
}
\newcommand{\reflect@squig}[2]{%
	\reflectbox{$\m@th#1\rightsquigarrow$}%
}
\makeatother

\newcommand{\svee}{\sigma^\vee}

\newcommand{\A}{\mathbb{A}}

\newcommand{\R}{\mathbb{R}}
\newcommand{\Z}{\mathbb{Z}}

\renewcommand{\P}{\mathbb{P}}

\newcommand{\cJ}{\mathcal{J}}

\newcommand{\cO}{\mathcal{O}}

\newcommand{\cS}{\mathcal{S}}
\newcommand{\cT}{\mathcal{T}}

\newcommand{\cZ}{\mathcal{Z}}
\newcommand{\cX}{\mathcal{X}}
\newcommand{\cY}{\mathcal{Y}}

\newcommand{\fa}{\mathfrak{a}}

\newcommand{\fm}{\mathfrak{m}}

\newcommand{\bA}{\mathbb{A}}
\newcommand{\bQ}{\mathbb{Q}}

\newcommand{\sX}{\mathscr{X}}

\newcommand{\inv}{^{-1}}

\newcommand{\cI}{\mathcal{I}}
\newcommand{\Bl}{\mathrm{Bl}}

\DeclareMathOperator{\Div}{div}

\DeclareMathOperator{\Spec}{Spec}

\DeclareMathOperator{\mult}{mult}

\DeclareMathOperator{\ord}{ord}

\DeclareMathOperator{\Proj}{Proj}

\newcommand{\Newt}{\operatorname{Newt}}

\newcommand{\w}{\omega}

\newcommand{\bft}{\mathbf{t}}

\newcommand{\tX}{\widetilde{X}}

\newcommand{\la}{\lambda}

\newcommand{\rT}{\cT}
\newcommand{\rS}{\cS}

\newcommand{\torus}{T}
\newcommand{\lattice}{L}
\newcommand{\duallattice}{M}

\newcommand{\dualcone}{\sigma^\vee}
\newcommand{\resdiv}{D}

\newcommand{\cOmpletion}{\widehat{\mathcal{O}}}
\newcommand{\basis}[1]{\mathbf{e}_{#1}}

\newcommand{\action}{\cdot}
\newcommand{\defeq}{:=}
\newcommand{\ra}{\rightarrow}

\newcommand{\Gm}{\mathbb{G}_m}

\newcommand{\Jf}{\left[\frac{J}{f}\right]}

\numberwithin{equation}{section}

\newtheorem{prop} {Proposition} [section]
\newtheorem{thm}[prop] {Theorem}

\newtheorem{claim}[prop] {Claim}

\newtheorem{prop-def}[prop]{Proposition-Definition}

\newtheorem{lemma}[prop]{Lemma}

\newtheorem{proposition}[prop]{Proposition}
\newtheorem{thm-defn}[prop]{Theorem-Definition}

\theoremstyle{definition}
\newtheorem{exa}[prop] {Example} 

\newtheorem{defn}[prop]{Definition}

\newtheorem{rmk}[prop]{Remark}
\newtheorem{App}[prop]{Application}

\theoremstyle{remark}

\newtheorem*{ThA*}{\textbf{Theorem A}}
\newtheorem*{ThB*}{\textbf{Theorem B}}
\newtheorem*{ThC*}{Theorem C}
\newtheorem*{ThD*}{Theorem D}
\newtheorem*{ThE*}{Theorem E}
\newtheorem*{Con*}{Conjecture}

\title{Multiplier Modules of extended Rees algebras}

\author{Rahul Ajit}
\address{Department of Mathematics\\
	University of Utah\\
	Salt Lake City, UT 84112, USA.}
\email{rahulajit@math.utah.edu}

\date{\today}

\begin{document}

	\begin{abstract}
		Given a local ring $(R, \fm)$ and an ideal $\fa$ of positive height, we give a way of computing multiplier module $\cJ(\w_T, t^{-\lambda})$ for the extended Rees algebra $T =R[\fa t, t^{-1}]$ for an ideal $\fa$ by proving a decomposition theorem for $\cJ(\w_T, t^{-\lambda})$, (also see \cite{Budur-Mustata-SaitoBS}). We compute the multiplier module $\cJ(\w_{\cS}, (\fa \cdot \cS)^{\lambda})$ for the Rees algebra $S =R[\fa t]$ as well, (also see \cite{HyryBlowUp} and \cite{KotalKummini}). We use these decompositions to understand relationships between associated graded rings, Rees and extended Rees algebras having rational singularities (also see \cite{HWY-F-Regular}).
	\end{abstract}
	
	\maketitle
    \begin{center}
   \textit{To Prof. Karen E. Smith, on the occasion of her 60th birthday!}
\end{center}
	 
	\setcounter{tocdepth}{1}
	
	\tableofcontents
	
	\newpage
	
\section{Introduction}
Let $(R,\fm)$ to be a local ring of dimension $d\geq 2$ and $\mathfrak{a}\subset R$ an ideal with positive height. We will be over a field $k$ of characteristic 0. We set $\cS=R[\mathfrak{a}t]$ to be the Rees algebra of $\mathfrak{a}$ with homogeneous maximal ideal $\fm_\cS =  \fm \oplus \mathfrak{a}t \oplus \mathfrak{a}^2t^2 \oplus \dots$ and $\cT=R[\mathfrak{a}t,t^{-1}]$ the extended Rees algebra of $\mathfrak{a}$ with homogeneous maximal ideal $\fm_\cT = \cdots \oplus Rt^{-2}\oplus Rt^{-1} \fm \oplus \mathfrak{a}t \oplus \mathfrak{a}^2t^2 \oplus \dots$  Let $G = \cT/(t^{-1}) = \bigoplus_{n \geq 0}\fa^n/\fa^{n+1}$ be the associated graded ring. Relationships between various properties (particularly, Cohen-Macaulayness and Gorensteinness) of $\cS, \cT$ and $G$ is a rich and extensively studied subject in commutative algebra and algebraic geometry, see \cite{GotoShimodaCM}, \cite{iai2024characterizationsgorensteinreesalgebras}, \cite{Goto-Nishida-Book}, \cite{Ikeda-Gor}, \cite{VietCM}, \cite{WhenCM}, \cite{WhenGor}, \cite{Huneke-82} and, \cite{Lipman1994CM} for example. Geometrically, Spec $\cT$ corresponds to the deformation to the normal cone Spec $G$, see \cite[Chapter 5]{Fulton}. A common philosophy is that if $G$ has some ring-theoretic property (such as reduced, integral, integrally closed, Cohen-Macaulay etc), so does
$R$, see \cite[Corollary 6.11]{Eisenbud-CA}, \cite{Goto-Nishida-Book}, \cite{Huneke-82}, \cite[Exercise 5.9]{Huneke-Swanson}.

Many authors studied similar questions for rational singularities, for example, \cite{HWY-F-Regular},\cite{HWY-02F-Rational} \cite{HyryCoef} and \cite{LipmanAdjoint}. In particular, Hara, Watanabe, and Yoshida proved, for $\fm$-primary ideal $\fa$ and ($F$-)rational ring R, $\cS$ is ($F$-)rational implies $\cT$ is ($F$-)rational, see \cite{Koley-Kummini21} and \cite{Ajit-Simper-Char-p} for the converse. By deformation, this implies, if  $G$ is ($F$-)rational then so are $\cS$ and, $R$, see \cite[page 182]{HWY-02F-Rational} and Section \ref{DeformationtotheNormalCone}.

\begin{ThA*}(Theorem \ref{Th-equivofFrationality}, cf.\cite{HWY-F-Regular})
 \textit{With the notations introduced above, we have the following:}
    \begin{enumerate}
        \item \textit{$R$ and $\cS$ both having rational singularities implies $\cT$ has rational singularities.}
        \item \textit{$\cT$ has rational singularities implies $R$ and $\cS$ both having rational singularities.}
    \end{enumerate} 
\end{ThA*}
We prove this theorem by proving a decompositions of multiplier modules of $\cS$ and $\cT$. 

\begin{ThB*}(Theorem \ref{theorem:multipliermodule-decomp}, cf. \cite{Budur-Mustata-SaitoBS})
\textit{Let $(R,\mathfrak{m})$ be a normal local ring of dimension at least $2$, finite type over a field of characteristic $0$. Let $\fa$ be an $R$-ideal with $\mathrm{ht}(\fa)>0$. Write $\cS = \bigoplus\limits_{n \ge0} \fa^nt^n$ for the Rees algebra of $\fa$ and 
$\cT := \bigoplus\limits_{n \in \Z} \fa^nt^n$ for the extended Rees algebra of $R$ with respect to $\fa$, where, $\fa^n := R$ for $n \leq 0$. Let $\lambda \geq 0$ be any real number.
Then} 
\begin{enumerate}
    \item $\cJ(\omega_{\cS},(\fa\cdot \cS)^\lambda) = \bigoplus\limits_{n\geq 0} \cJ(\omega_R, \fa^{n+1+\lambda})t^{n+1}.$\label{ReesMultiplier}

    \item $\cJ(\omega_{\cT}, (t^{-1})^\lambda) = \bigoplus\limits_{k \in\Z} \cJ (\omega_R, \fa^{k+\lambda})t^k.$ \footnote{With the convention that when $k+  \lambda  \leq 0$, we define $\fa^{k+  \lambda } = R$ }\label{ExtReesMultiplier}
\end{enumerate}

\textit{In particular, $[\cJ(\omega_{\cT}, (t^{-1})^\lambda)]_0 =  \cJ (\omega_R, \fa^{\lambda})$ and $[\cJ(\omega_{\cT}, (t^{-1})^\lambda)]_{\geq 1} = \cJ(\omega_{\cS},(\fa\cdot \cS)^\lambda)$}.
\end{ThB*}

Similar decompositions for the multiplier module of the Rees algebra $\cS$ under more restricted setups appeared in many other works. For example, in \cite[Proposition 3.1]{HyryBlowUp}, Hyry computed a similar formula for multiplier modules $\cJ(\w_S) $ when $R$ is regular and $X = \Proj \cS$ has rational singularities. In \cite[Theorem 5.1]{HaraYoshida}, over a field of positive characteristics, Hara and Yoshida established very similar formula for test modules $\tau(\w_\cS)$ (which measure $F$-rationality) under the assumptions that $R$ is Gorenstein, $\fa$ is $\fm$-primary, and $\cS$ is $F$-rational. Kotal and Kummini generalized \cite{HaraYoshida} and \cite{HyryBlowUp} to the case when $R$ and $\cS$ are normal, Cohen-Macaulay, see \cite[Theorems 1.1 and 1.3]{KotalKummini}.

A similar decomposition for the multiplier ideals of the extended Rees algebra $\cT$ is obtained by Budur, Mustaţă and Saito, under the assumption that $\Spec \ R$ is smooth, using V-filtrations \footnote{After finishing writing this article, Bradley Dirks communicated to me that inspired by Theorem \ref{theorem:multipliermodule-decomp}, he also proved an exact similar decomposition for the multiplier modules of the extended Rees algebra $\cT$ using V-filtrations in his upcoming paper \cite{Brad}.}, see \cite[Theorem 1 along with (1.3.1)]{Budur-Mustata-SaitoBS}.

Similar decompositions for test modules in characteristic $p$ have been proved by Hunter Simper and the author in \cite[Theorem 5.2]{Ajit-Simper-Char-p}, using very different techniques.

\section{Acknowledgments}

I would like to thank my advisors, Christopher Hacon and Karl Schwede, for their constant encouragement, unwavering support, inspiring teachings, and infinite patience. I am grateful to Harold Blum for a lot of stimulating and extremely helpful discussions, which played a crucial role in this paper. I would like to thank Bradley Dirks for informing me about \cite{Budur-Mustata-SaitoBS} and \cite{Brad}. Finally, I would like to thank Daniel Apsley, Harold Blum, Christopher Hacon and Karl Schwede for reading this article and providing many valuable feedback and helpful comments, which improved the exposition substantially. I was partially supported by NSF research grant DMS-2301374 and by a grant from the Simons Foundation SFI-MPS-MOV-00006719-07 while working on this project.

\section{Basic definitions}

\begin{defn} (\cite[Definition 2.1]{WeakOrdanityConj}, \cite{RationalSingPair}, \cite{Blickle})
Suppose that $X$ is a normal variety over a field of characteristic zero, $\Delta$ is an effective $\bQ$-divisor, $\fa$ is an ideal sheaf and $\la \geq 0$ is a real number.
Let $\pi : Y \to X$ be a proper birational morphism with $Y$ normal such that $\fa \cdot \cO_Y = \cO_Y(-G)$ is invertible, and we assume that $K_X$ and $K_Y$ agree wherever $\pi$ is an isomorphism.
\begin{itemize}
\item[(a)]  If $K_X + \Delta$ is $\bQ$-Cartier, we assume that $\pi$ is a log resolution of $(X, K_X + \Delta, \fa)$.  Then we define the \textbf{multiplier ideal} to be
\[
\hspace*{2em} \cJ(X, \Delta, \fa^\la) = \pi_* \cO_Y(\lceil K_Y - \pi^*(K_X + \Delta) - \la G\rceil) \subseteq \cO_X.
\]
\item[(b)]  If $\Delta$ is $\bQ$-Cartier, we assume that $\pi$ is a log resolution of $(X, \Delta, \fa)$.  Then we define the \textbf{multiplier module} to be
\[
\cJ(\omega_X, \Delta, \fa^\la) = \pi_* \cO_Y(\lceil K_Y - \pi^*\Delta  - \la G\rceil) \subseteq \omega_X.
\]
\end{itemize}
When $\Delta = 0$, we omit it, i.e, we write $\cJ(\omega_X, \fa^\la) := \cJ(\omega_X, 0, \fa^\la)$
\end{defn}
It is a general fact that these definitions are independent of the (log) resolution chosen, see \cite[Theorem 9.2.18]{Lazarsfeld-Positivity-2}.

\begin{defn}
Let $(X, \fa^\la)$ be a pair and let $\pi : Y \rightarrow X$ with $\fa \cO_{Y} = \cO_{Y}(-G)$ be a log resolution of $\fa$.  We say that the pair $(X, \fa^\la)$ has \textbf{rational singularities}  if the natural map $\cO_X \rightarrow \mathbf{R} \pi_* \cO_{Y}(\lfloor \la G \rfloor)$ is a quasi-isomorphism.

A reduced Noetherian ring $R$ has \textbf{rational singularities} if $\operatorname{Spec}R$ has rational singularities.
\end{defn}

\begin{lemma}\cite[Lemma 2.3]{WeakOrdanityConj}
\label{PropertiesOfMultiplier}
With notation as above:
\begin{itemize}
\item[(i)]  $\cJ(X, \Delta, \fa^\la) = \cJ(\omega_X, K_X + \Delta, \fa^\la)$.
\item[(ii)]  If $D$ is a Cartier divisor, then
\[
\cJ(X, \Delta +D, \fa^\la) = \cJ(X, \Delta, \fa^\la) \otimes \cO_X(-D).
\]
\end{itemize}
\end{lemma}

\begin{rmk}
    Suppose that $X$ is a reduced equidimensional Cohen-Macaulay scheme and $\fa$ is an ideal sheaf on $X$.  Then $(X, \fa^\la)$ has rational singularities if and only if the multiplier submodule $\cJ(\omega_X, \fa^\la)$ is equal to $\omega_X$, see \cite[Corollary 3.8]{RationalSingPair}.
\end{rmk}

\subsection{Grothendieck's natural construction}\label{GrothNatural} We will follow \cite[Remark 2.8]{Hyry-Villamayor}, \cite[Section 6.2.1]{HyrySmith} and, \cite[Section 4]{AjitGR-solo} closely. Let $(R, \fm)$ be a local ring, and let $\cI \subset R$ be an ideal. Consider the blow-up 
\[
f: X = \Proj \ S =\Proj\left(\bigoplus_{n \geq 0} \cI^n\right) \longrightarrow \Spec R
\] 
along $\cI$.Consider the graded $S$-algebra defined by:
\[
S^\# := S \oplus S_{\geq 1} \oplus S_{\geq 2} \oplus \cdots,
\]
where $S_{\geq n} = \bigoplus_{m \geq n} \cI^m$ for each integer $m \geq 0$. Define the scheme
\[
Y := \Proj(S^\#).
\]

The scheme $Y$ admits two geometric interpretations:

\begin{enumerate}
\item \textbf{As a blow-up:} The algebra $S^\#$ is naturally isomorphic to the Rees algebra of the ideal $S_{\geq 1} \subset S$. So, the natural projection
\[
\theta: Y \longrightarrow \Spec S
\]
is the blow-up of $\Spec S$ along $S_{\geq 1}$.

\item \textbf{Total space of a line bundle:} There exists a canonical isomorphism of $X$-schemes
\[
Y \cong \Spec_X \left( \bigoplus_{m \geq 0} \cO_X(m) \right).
\]
This isomorphism identifies $Y$ with the total space of the line bundle $\cO_X(-1)^\vee$, which is the dual of the tautological line bundle associated with the blow-up $X = \Proj S$ and $\cO_X(1) \simeq \cI \cdot \cO_X$.  Let
\[
\eta: Y \longrightarrow X
\]
denote the corresponding bundle map.
\end{enumerate}

We have the following commutative diagram
\[
\begin{tikzcd}
Y \arrow[r, "\theta"] \arrow[d, "\eta"'] & \Spec S \arrow[d, "\pi"] \\
X \arrow[r, "f"'] & \Spec R
\end{tikzcd}
\]
where $\pi: \Spec S \to \Spec R$ is the natural affine morphism induced by the inclusion $R \hookrightarrow S$. The morphism $\eta: Y \to X$ is smooth of relative dimension 1. The relative cotangent bundle $\Omega_{Y/X} \cong \eta^* \cO_X(1)$. So, we have
\[
\omega_Y \cong \eta^* \omega_X \otimes \Omega_{Y/X} \cong \eta^* \omega_X \otimes \eta^* \cO_X(1) \cong \eta^*(\omega_X(1)).
\]
Since $\eta$ is affine
\[
\eta_* \omega_Y \cong \eta_* \left( \eta^* (\omega_X(1)) \right) \cong \omega_X(1) \otimes \eta_* \cO_Y.
\]

As $Y$ is the total space of $\cO_X(-1)$, we have
\[
\eta_* \cO_Y \cong \bigoplus_{m \geq 0} \cO_X(m).
\]
Therefore
\[
\eta_* \omega_Y \cong \bigoplus_{m \geq 0} \omega_X(1) \otimes \cO_X(m) \cong \bigoplus_{m \geq 0} \omega_X(m+1) \cong \bigoplus_{m \geq 0} \cI^{m+1} \omega_X. 
\]
As $\eta$ is affine,
\[
H^i(Y, \omega_Y) \cong H^i(X, \eta_* \omega_Y) \cong \bigoplus_{m \geq 0} H^i(X, \cI^{m+1} \omega_X) \quad \text{for all } i \geq 0.
\]
.

\section{Deformation to the normal cone}\label{DeformationtotheNormalCone}

The classical reference for this section is \cite[Chapter 5]{Fulton}. However, we will closely follow \cite[13.4.1]{Eisenbud-3264} and Vakil's Intersection Theory course notes \cite[Section 3]{DNCVakil}.
Let $(R, \fm)$ be a local ring and $I$ be an ideal of $R$. Define $\cT = \bigoplus_{n \in \mathbb{Z}} I^n t^n \subset R[t, t^{-1}]$, with $I^{n} = R$ for $n \le 0$ to be the extended Rees Algebra. Note that $\cT$ is a $k[t^{-1}]$-algebra. Let us start by recording the following easy facts.

\begin{lemma} Under the above setup, the following holds:
    \begin{enumerate}
    \item (\cite[6.5]{Eisenbud-CA}) $\cT/(t^{-1}) \cong \bigoplus_{n \geq 0} I^n / I^{n+1} =  G$, the associated graded ring.
    \item \cite[6.5]{Eisenbud-CA} $\cT/(t^{-1} - a) \cong R$ for any unit $a \in R^\times$.
    \item $\cT_{t^{-1}} \cong R[t, t^{-1}]$ 
    \item (\cite[Corollary 6.11]{Eisenbud-CA}) $\cT$ is flat over $k[t^{-1}]$.
\end{enumerate}
\end{lemma}

\subsection{Affine deformation space}
For $X = \Spec R/I \subseteq Y = \Spec R$, define:
\[
\mathscr{M} := \Spec \ \cT, \quad \tau: \mathscr{M} \to \mathbb{A}^1 = \Spec k[t^{-1}]
\]

\begin{lemma}
    \begin{enumerate}
    \item \textit{General fiber at $t^{-1} = a \neq 0$}, $a \in k^*$: 
    \[
    \tau^{-1}(a)  \cong \Spec \ R
    \]
    \item \textit{Special fiber ($t^{-1} = 0$)}:
    \[
    \tau^{-1}(0)  \cong \Spec \ G =: C_XY
    \]
\end{enumerate}
\end{lemma}

The above says that $G$ can be thought of as the special fiber of $\cT$, whereas a general fiber is $R$. This has very important consequence: if $G$ has some ring-theoretic property (such as reduced, integral, integrally closed, Cohen-Macaulay etc), so does
$R$, see \cite[Corollary 6.11]{Eisenbud-CA}, \cite{Goto-Nishida-Book}, \cite{Huneke-82}, \cite[Exercise 5.9]{Huneke-Swanson}.

\subsection{Global construction} (See \cite[13.4.1]{Eisenbud-3264}, \cite[Section 4]{DNCVakil} for details)
Blow-up $Y \times \P^1$ at $X \times \{0\}$ and define $\mathscr{M}_{\P^1} := \Bl_{X \times \{0\}} (Y \times \P^1)$ with projection $\tau: \mathscr{M}_{\P^1} \to \P^1$.

\begin{lemma} (see \cite[13.4.1]{Eisenbud-3264}, \cite[Page 4]{DNCVakil})\label{ExceptionalDNC}
    The $\tau$-exceptional divisor satisfies $E_{X \times \{0\}}Y\times \P^1 \cong \P (C_XY \oplus \mathbf{1})$, where the later is the projective compactification of the total space of the normal bundle $C_XY$.
\end{lemma}

\begin{lemma} (\cite[13.4.1]{Eisenbud-3264},\cite[3.1]{DNCVakil})
    The scheme-theoretic fiber $\tau^{-1}(0)$ consists of two irreducible components: the proper transform $\tX\cong \Bl_X Y$ of $X \times \{0\}$ and the exceptional divisor $E_{X \times \{0\}}Y\times \P^1 \cong \P (C_XY \oplus \mathbf{1})$, glued along the ``hyperplane at infinity" $\mathbb{P}(C_XY)$. That is, $\tau^{-1}(0) \cong \Bl_X Y \bigcup_{\mathbb{P}(C_XY)} \mathbb{P}(C_XY \oplus \cO_X)$ .
\end{lemma}

\subsection{Deformation to the normal cone}
Define
\[
\mathscr{M}^\circ := \mathscr{M}_{\mathbb{P}^1} \setminus \Bl_X Y
\]

\begin{proposition} (\cite[Theorem 13.8]{Eisenbud-3264}, \cite[Section 4]{DNCVakil})The morphism $\tau': \mathscr{M}^\circ \to \mathbb{P}^1$ is flat and

\begin{enumerate}
    \item $\tau'^{-1}(t) \cong Y$ for $t \neq 0$,
    \item $\tau'^{-1}(0) = \mathbb{P}(C_XY \oplus \cO_X) \setminus \mathbb{P}(C_XY) \cong C_XY$
  
\end{enumerate}
\end{proposition}

\section{Main steps of the proof}

Let $(R,\fm)$ be a local ring of dimension $d\geq 2$ and $\mathfrak{a}\subset R$ an ideal with positive height and $\cT=R[\mathfrak{a}t,t^{-1}]$ the extended Rees algebra of $\mathfrak{a}$ with homogeneous maximal ideal $\fm_T = \cdots \oplus Rt^{-2}\oplus Rt^{-1} \fm \oplus \mathfrak{a}t \oplus \mathfrak{a}^2t^2 \oplus \dots$ . Here we explain the main steps of our multiplier module formula for the pair $(\cT, t^{-\lambda})$. We start with the diagram:

\[\begin{tikzcd}
	{\ \sX=\Spec_X\left(\bigoplus\limits_{n\in\Z}\fa^n\cO_X t^n\right)}       & X \\
	{\Spec T} & {\Spec R}
	\arrow["{\pi}", from=1-1, to=1-2]
	\arrow["\beta"', from=1-1, to=2-1]
	\arrow["\psi", from=1-2, to=2-2]
	\arrow[from=2-1, to=2-2]
    \end{tikzcd}\]
where $\psi : X \longrightarrow \Spec \ R $ is a log resolution of singularity of $(\Spec \ R, \fa).$ All other maps are defined naturally.

\begin{lemma}
The scheme $\mathscr{X}$ is Cohen-Macaulay.
\end{lemma}

\begin{proof}
Locally, for an affine open $U \subset X$ where $\mathfrak{a}|_U = (f)$ for $f \in \mathcal{O}_X(U)$. Then
\[
\mathscr{X}|_U =\Spec \ \cO_U[ft, t^{-1}] = \operatorname{Spec} \mathcal{O}_U[x,y]/(xy - f).
\]
This is a hypersurface in $\mathbb{A}^{n+1}_U$. Hence it is Cohen-Macaulay.
\end{proof}

\begin{lemma}\label{normality}
$\mathscr{X}$ satisfies Serre's condition $R_1$ \cite[\href{https://stacks.math.columbia.edu/tag/033P}{Tag 033P}]{stacks-project} i.e., $\operatorname{codim}(\operatorname{Sing}(\mathscr{X})) \geq 2$.
\end{lemma}

\begin{proof}
 The scheme $\mathscr{X}|_U = \operatorname{Spec} \mathcal{O}_U[x,y]/(xy - f)$ \'etale locally isomorphic to $\Spec \ B$ where \(B = k[x, y, s_1, \dots, s_n] / (xy - u \prod_{i=1}^m s_i^{a_i})\). The singular locus is defined by the Jacobian ideal:
\[
J = \left( y, x, \frac{\partial f}{\partial s_1}, \dots, \frac{\partial f}{\partial s_n} \right).
\]
for $f = u \cdot s_1^{a_1} \cdots s_m^{a_m}$.
Hence, \(\text{Sing}(X) \subseteq V(x,y)\). Now,
\[
B/(x,y) \cong k[s_1, \dots, s_n] / \left( \prod_{i=1}^n s_i^{a_i} \right),
\]
which has dimension \(n - 1\). Since \(\dim B = n + 1\), the singular locus has codimension 2. Thus, \(B\) is $R_1$.
\end{proof}

\begin{rmk}
    The two lemmas above prove that $\sX$ is normal, \cite[\href{https://stacks.math.columbia.edu/tag/033P}{Tag 033P}]{stacks-project}. In particular, we can talk about $K_\sX$, see \cite[Corollary 1.18]{Schwede-Generalized}.
\end{rmk}

In the remainder of the paper we will prove Theorem B (Theorem \ref{theorem:multipliermodule-decomp}) through the following steps:

\begin{itemize}
    \item We will compute the canonical module $\w_{\sX}$.
    \item We will prove that for a closed point $\xi \in \sX$, the completion $\widehat{\cO_{\sX, \xi}}$ is isomorphic to completion $\cO_{\cX, 0}$ of an affine toric variety $\cX$.
    \item We will rephrase \cite[Theorem 2]{Blickle}, to use in our setup.
    \item Finally, we will use the above diagram to compute multiplier module $\cJ(\w_\rT , t^{-\lambda})$.
    
\end{itemize}

\section{Canonical module computation}
\subsection{Discrepancy under \'etale morphisms}We first prove invariance of log discrepancies under \'etale morphisms. This is very well-known, but we include a proof for only for completeness.
\begin{proposition}\label{DiscrepEtale}
Let $\phi: U \to X$ be an étale morphism of smooth varieties over a field $k$ of characteristic zero. Let $\fa \subseteq \cO_X$ be an ideal sheaf, and $E$ a prime divisor over $X$. Let $F$ be a divisor over $U$ dominating $E$ under $\phi$. Then:
\[
a(F; U, \phi^*\fa) = a(E; X, \fa).
\]
\end{proposition}

\begin{proof}
Consider the fiber product diagram:
\[
\begin{tikzcd}
V \arrow[r, "\psi"] \arrow[d, "g"] & Y \arrow[d, "h"] \\
U \arrow[r, "\phi"] & X
\end{tikzcd}
\]
where $h: Y \to X$ is a log resolution of $(X, \fa)$ with $\fa \cdot \cO_Y = \cO_Y(-G)$ where G is some effective Cartier divisor on Y, and $V = U \times_X Y$. Since $\phi$ is étale, $\psi$ and $g$ are étale.

The relative canonical divisors satisfy
\[
\psi^* K_{Y/X} = K_{V/U},
\]
where $K_{Y/X} = K_Y - h^* K_X$ and $K_{V/U} = K_V - g^* K_U$. This holds because:
\begin{align*}
\psi^* \Omega_{Y/X} &\cong \Omega_{V/U} \quad \text{(by étale base change)}, \\
\psi^* \omega_{Y/X} &\cong \omega_{V/U}.
\end{align*}

Let $K_Y = h^*(K_X + \Div(\fa)) + a(E)E + \sum b_i D_i$. Pulling back by $\psi$, we get
\[
\psi^* K_Y = \psi^* h^* (K_X + \Div(\fa)) + a(E) \psi^* E + \sum b_i \psi^* D_i.
\]
Since $\psi^* h^* = g^* \phi^*$,
\[
\psi^* K_Y = g^* \phi^* (K_X + \Div(\fa)) + a(E) \psi^* E + \sum b_i \psi^* D_i.
\]
Substituting $K_V = \psi^* K_Y$ (note, étale $\Rightarrow$ unramified),
\[
K_V = g^* \phi^* (K_X + \Div(\fa)) + a(E) \psi^* E + \sum b_i \psi^* D_i.
\]

Then,
\[
K_V - g^* K_U = g^* \phi^* \Div(\fa) + a(E) \psi^* E + \sum b_i \psi^* D_i.
\]
Note that $\phi^* K_X = K_U$  so,
\[
K_V - g^* K_U = g^* (\phi^* \Div(\fa)) + a(E) \psi^* E + \sum b_i \psi^* D_i.
\]
Thus
\[
K_{V/U} = g^* \Div(\phi^* \fa) + a(E) \psi^* E + \sum b_i \psi^* D_i.
\]

Let $\eta_E$ and $\eta_F$ be generic points of $E$ and $F$. Since $\psi$ is étale and $F$ dominates $E$, $\widehat{\cO}_{Y,\eta_E} \cong \widehat{\cO}_{V,\eta_F}$. Thus $\psi^* E = F$ near $\eta_F$. Therefore
\[
K_{V/U} = g^* \Div(\phi^* \fa) + a(E) F + \sum b_i \psi^* D_i,
\]
which implies $a(F; U, \phi^* \fa) = a(E; X, \fa)$.
\end{proof}

\begin{rmk}\label{MultiplierEtale}
    Proposition \ref{DiscrepEtale} immediately tells us that the multiplier module remains unchanged under smooth \'etale morphisms, see \cite[Example 9.5.44]{Lazarsfeld-Positivity-2}.
\end{rmk}

\subsection{Setup}
Let $k$ be a field of characteristic zero. Suppose the following are given:
\begin{itemize}
    \item $X $ a smooth variety of dimension $n-1$
    \item $\fa\cdot \cO_X = (-E) \subset \cO_X$ where $E$ has  SNC support.
    \item $\sX = \operatorname{Spec}_X \left( \bigoplus_{k \in \mathbb{Z}} \fa^k  \cdot \cO_X t^{k}\right)$ with convention that $\fa^{-k}:= A$ for $k \geq 0$. Note, $\sX$ is isomorphic to \(\Bl_{(\fa,t^{-1})}(X \times \A^1) \setminus \widetilde{V(t^{-1})}\), where \(\widetilde{V(t^{-1})}\) is the strict transform of \(X \times \{0\}\). The proof follows by looking at the affine charts $D_+(f)$ and $D_+(t^{-1})$. Gluing over $D_+(f t^{-1})$ gives global isomorphism, see Section \ref{DeformationtotheNormalCone}.
    \item $\mathcal{Y} = \Bl_{(\fa,t^{-1})}(X \times \A^1)$ is the blow-up of $X \times \A^1$ at $(\fa,t^{-1})$ and $\sX \subseteq \cY$ is an open immersion.
\end{itemize}
We want to compute the canonical divisor on $\cY$ first and then, restrict the canonical divisor $K_{\cY}$ to $\sX$ to get $K_{\sX}$. This reduces to computing the log discrepancy $A_{X\times\bA^1}(E)$ for each irreducible component $E \subset \sX_0 = V(t^{-1})$. Note that we already know the canonical divisor $K_{X\times \bA^1} = p_1^*K_X + p_2^*K_{\A^1}$. 
We have the following diagram:

\[
\begin{tikzcd}
                                                                           &  & \mathcal{Y} \arrow[rrdd, "\beta"] &  &                                                      &  &   \\
                                                                           &  &                                   &  &                                                      &  &   \\
\mathscr{X} \arrow[rruu, "\alpha", hook] \arrow[rrdd, "\tau"] \arrow[rrrr] &  &                                   &  & X \times \A^1 \arrow[lldd, "p_2"] \arrow[rr, "pr_1"] &  & X \\
                                                                           &  &                                   &  &                                                      &  &   \\
                                                                           &  & \A^1                              &  &                                                      &  &  
\end{tikzcd}
\]
where $\alpha : \sX \longrightarrow \cY$ is the open immersion, $\beta :\mathcal{Y} = \Bl_{(\fa,t^{-1})}(X \times \A^1) \longrightarrow X \times \A^1$ is the blow-up along $(\fa,t^{-1})$, and $\tau : \sX \longrightarrow \A^1$ is flat, see Section \ref{DeformationtotheNormalCone}. 
\subsection{Discrepancy computation} We prove the following proposition, under the above setup:
\begin{proposition}\label{DiscrepancyComputation} \hfill
    \begin{enumerate}
        \item There is a bijection between the irreducible components of $\sX_0$ and irreducible components of $E$. In particular, if $E = \sum a_iE_i$ where $E_i$'s are smooth and irreducible, then same is true for $\sX_0 = \sum a_iG_i$ where $G_i = (\beta \circ \alpha)^{-1}(E_i)$
        \item For every irreducible component $G \subseteq \sX_0$, we have $a(G; X\times \A^1) = \mult_G(\sX_0)$.
    \end{enumerate}
\end{proposition}

\begin{proof}
Fix and affine open $U\subset X$ such that $\fa\cdot\cO_U = (f)$ where $f =  s_1^{a_1}\cdots s_m^{a_m}$. Then, $$\sX|_U \simeq \Spec \ \cO_U[ft, t^{-1}]\simeq \Spec \ \cO_U[x, t^{-1}]/(xt^{-1}-s_1^{a_1}\cdots s_m^{a_m}).$$ So, $$(\sX|_U)_0 \simeq\Spec \ \cO_U[x]/(s_1^{a_1}\cdots s_m^{a_m}) \simeq \Spec \ ((\cO_U/(s_1^{a_1}\cdots s_m^{a_m}))[x]).$$  This proves part $(1)$. 

For part $(2)$, we will do several reductions:

\begin{itemize}
 \item Note that $\sX$ is the blowup of $X\times \A^1$ along $(f,t^{-1})$ with the strict transform of $X\times \{0\}$ removed. So, for any irreducible component $G_i \subset \mathscr{X}_0 = V(t^{-1})$, the discrepancy only depends on the valuation $\ord_{G_i}$ which can be computed on some neighborhood of the generic point of $G_i$ (see \cite[Remark 2.23]{KollarMori98}). By part $(1)$ we can assume it to be centered at $\eta \times \{0\} \in X \times \mathbb{A}^1$, where $\eta$ is the generic point of the corresponding divisor $E_i \subset X$. 
\item  As $U \subset X$ is smooth and $V(f)$ is SNC divisor, around the generic point $\eta$ of an irreducible component of $V(f)$, we can find $V \subset U$ such that $\rho:V \longrightarrow \A^{n-1} = \Spec$ $k[x_1, \cdots , x_{n-1}]$ is \'etale and $f$ corresponds to $x_1^d$, by \cite[054L]{stacks-project}. 
\item For \'etale morphism between smooth varieties, log discrepancy remains unchanged by Proposition \ref{DiscrepEtale}.

    \item So, \'{e}tale locally, we can assume that $X= \bA^{n-1}= \Spec k[x_1,... , x_{n-1}]$. Thus we are blowing up $(x_1^d,t)$ in $\bA^{n-1} \times \bA^1 = \bA^n$ and all we need to do is to compute the log discrepancy of the exceptional divisor.
    \item Furthermore, by taking a product, this reduces to understanding the log discrepancy of the blowup of $\bA^1 \times \bA^1 = \A^2$ along $(x_1^d, t)$ which is precisely what \cite[6.39.1]{CortiKollarSmith} computes. We do this in the next \textbf{Claim}.
\end{itemize}

\begin{claim}
Let \(W = \A^{n-1} \times \A^1 = \A^n\), \(Z = V(x_1^d, t)\), $\pi: \widetilde{W} \longrightarrow W$ be the blow up at Z with exceptional divisor $G$. Then,
\[
a(G; W) = d.
\]
\end{claim}

\begin{proof}
Let \(S = V(x_2, \dots, x_{n-1}) \cong \A^2\). The inclusion \(S \hookrightarrow W\) is transverse to \(Z\). Since product commutes with blow up, we have the following isomorphisms $$\widetilde{W} = \mathrm{Bl}_{Z} W  \simeq \mathrm{Bl}_{Z \cap S} S \times \mathbb{A}^{n-2} = \widetilde{S} \times \mathbb{A}^{n-2}  $$ where $(\widetilde{S} = \Bl_{Z \cap S} S).$

The blowup commutes with restriction, by \cite{Hartshorne}, $\widetilde{S} = \Bl_{Z \cap S} S \cong \mathrm{strict\ transform \ of \ } S \subset \widetilde{W}$.

The canonical divisors satisfy
\[
K_{\widetilde{W}} = \operatorname{pr}_1^* K_{\widetilde{S}} + \operatorname{pr}_2^* K_{\A^{n-2}}
\]
\[
K_W = \operatorname{pr}_1^* K_S + \operatorname{pr}_2^* K_{\A^{n-2}}
\]
By \cite[6.39.1]{CortiKollarSmith}, \(K_{\widetilde{S}} = \pi_S^* K_S + d G_S\), where \(G_S = E \cap \widetilde{S}\). Then
\[
K_{\widetilde{W}} = \pi^* (\operatorname{pr}_1^* K_S + \operatorname{pr}_2^* K_{\A^{n-2}}) + d G = \pi^* K_W + d G
\]
So \(a(G; W) = a(G_S; S)\) = $d$.
\end{proof}
Thus we get $a(G; X\times \A^1) = \mult_G(\sX_0)$.
\end{proof}

\subsection{Canonical module $\w_\sX$}\label{CanonicalComputation}
Note that $\sX \subset \mathcal{Y}$ is the complement of the strict transform of $X \times \{0\}$. The map $\tilde{\pi}: \sX \to X$ denotes the composition
\[
\sX \hookrightarrow \mathcal{Y} \xrightarrow{\pi} X \times \mathbb{A}^1 \xrightarrow{p_1} X
\]

Next note, \[
\w_{X \times \A^1} \simeq \w_X \boxtimes \w_{\A^1} \simeq \cO_X(K_X) \boxtimes \cO_{\A^1}(-0) \simeq \cO_{X\times \A^1}(pr_1^*K_X - X \times {0})
\]

So, from Proposition \ref{DiscrepancyComputation} we get

\[
\w_{\sX} = (\w_\cY)|_\sX= {\pi}^* \w_{X\times \A^1} \otimes \cO_{\cY}(\cY_0 - \widetilde{X\times 0})|_{\sX} = (\pi|_{\sX})^*\w_{X \times \A^1} \otimes \cO_{\sX}(\sX_0) = (\pi|_{\sX})^*\cO_{X\times \A^1}(pr_1^*K_X) = \tilde{\pi}^*(\w_X)
\]where $\cY_0 = V(t^{-1}).$

Pushing forward and using projection formula,  we get,

\[
\tilde{\pi}_* \w_{\sX} = \bigoplus_{n \in \mathbb{Z}}(\fa^{n} \cdot \w_X)t^{n} \label{canonialDNC}
\] with convention that $\fa^k \cdot \cO_X = \cO_X, \forall k \leq 0.$

\begin{rmk}
    Similar computation of canonical modules have appeared in the past. For example, see \cite[Corollary 3.3]{TomariWatanabe} when $height(\fa) \geq 2$ and \cite[Appendix B]{MSTWW} when $\fa$ is $\fm$-primary. In \cite{Ajit-Simper-Char-p}, the author of this article, along with Hunter Simper, proved a general result relating canonical modules of the extended Rees algebra $\cT$ and the Rees algebra $\cS$, for any ideal $\fa$ with $height(\fa) > 0$.
\end{rmk}

\section{Blickle's formula for multiplier modules}
We first collect the following results about Toric Varieties. The standard reference is \cite{ToricBook}.
\begin{defn}
Let $N \cong \Z^n$ be a lattice. A \textbf{rational polyhedral cone} $\sigma \subseteq N_\R = N \otimes_\Z \R$ is a set
\[
\sigma = \{ r_1 v_1 + \cdots + r_s v_s \mid r_i \geq 0 \}
\]
for vectors $v_1, \dots, v_s \in N$ (the \textbf{generators}). The cone is \textbf{strongly convex} if $\sigma \cap (-\sigma) = \{0\}$ (contains no lines). It is \textbf{rational} because the $v_i$ have integer coordinates.
\end{defn}

\begin{defn}
The \textbf{dual lattice} $M = \operatorname{Hom}(N, \Z) \cong \Z^n$ has a pairing $\langle \cdot, \cdot \rangle: M \times N \to \Z$. The \textbf{dual cone} $\svee \subseteq M_\R$ is
\[
\svee = \{ m \in M_\R \mid \langle m, n \rangle \geq 0 \text{ for all } n \in \sigma \}.
\]
The \textbf{interior} of $\svee$ is
\[
\mathrm{int}(\svee) = \{ m \in M_\R \mid \langle m, n \rangle > 0 \text{ for all } n \in \sigma \setminus \{0\} \}.
\]
\end{defn}

\begin{defn}
Given a cone $\sigma \subseteq N_\R$, the coordinate ring
\[
R_\sigma = k[\svee \cap M] = k \text{-span}\{x^m \mid m \in \svee \cap M\}
\]
defines the \textbf{affine toric variety} $X_\sigma = \Spec(R_\sigma)$. Monomials $x^m$ correspond to lattice points $m \in \svee \cap M$.
\end{defn}

\begin{exa}

For $N = \Z^2$, $\sigma = \operatorname{cone}\{(1,0), (0,1)\}$ (first quadrant), then $\svee = \sigma$, and $X_\sigma = \mathbb{A}^2_{k}$ with coordinate ring $k[x,y]$.
\end{exa}

\begin{defn}
Let $v_1, \dots, v_s$ be the primitive generators of the rays of $\sigma$ (minimal lattice points). These define \textbf{prime Weil divisors} $D_1, \dots, D_s$ on $X_\sigma$. Any torus-invariant Weil divisor is $D = \sum_{i=1}^s a_i D_i$ with $a_i \in \Z$.
\end{defn}

\begin{defn}
The \textbf{canonical divisor} is $K_X = -\sum_{i=1}^s D_i$.
\end{defn}

\begin{defn}
For a monomial $t = x^u$ ($u \in \svee \cap M$), its divisor is
\[
\Div(x^u) = \sum_{i=1}^s \langle u, v_i \rangle D_i.
\]
The pairing $\langle u, v_i \rangle \in \Z_{\geq 0}$ since $u \in \svee$ and $v_i \in \sigma$.
\end{defn}

\begin{defn}
For a Weil divisor $D = \sum_{i=1}^s a_i D_i$, the sheaf $\cO_X(D)$ consists of rational functions $f$ such that $\Div(f) + D \geq 0$. For monomials
\[
x^v \in \Gamma(X, \cO_X(D)) \iff \langle v, v_i \rangle + a_i \geq 0 \quad \forall i = 1, \dots, s.
\]
\end{defn}

\begin{defn}
The \textbf{canonical module} $\omega_X \subseteq R_\sigma$ is the ideal
\[
\omega_X = \{ x^m \mid m \in \mathrm{int}(\svee) \} = \langle x^m \mid \langle m, v_i \rangle > 0 \text{ for all } i \rangle.
\]
\end{defn}

\begin{defn}
For a monomial ideal $\mathfrak{a} \subseteq R_\sigma$, its \textbf{Newton polyhedron} $\Newt(\mathfrak{a})$ is the convex hull in $M_\R$ of the set $\{ m \in M \mid x^m \in \mathfrak{a} \}$. For a principal ideal $\mathfrak{a} = (x^u) $
\[
\Newt((x^u)) = u + \svee = \{ u + m \mid m \in \svee \}.
\]
\end{defn}

\begin{thm} \cite{Blickle}
 Let X be an affine toric variety and $\fa$ be a monomial ideal. Then
\[
\cJ(\omega_X, \mathfrak{a}^\lambda) = \langle x^v \mid v \in \mathrm{int}(\lambda \cdot \Newt(\mathfrak{a})) \rangle \subseteq \omega_X.
\]
\end{thm}

\begin{lemma}\label{BlickleModule}
    Additionally, in the above situation, when $\fa = (x^u)$ is principal, we have \[ \cJ(\omega_X, \mathfrak{a}^\lambda) = \cO_X(\lceil K_X - \lambda \Div(x^u) \rceil)\]
\end{lemma}

\begin{proof}
   
For $\mathfrak{a}  = (x^u)$, we have $\Newt((x^u)) = u + \svee$. Scaling by $\lambda > 0$
\[
\lambda \cdot \Newt((x^u)) = \lambda(u + \svee) = \lambda u + \lambda \svee.
\]
The interior is $\mathrm{int}(\lambda \cdot \Newt((x^u))) = \lambda u + \lambda \mathrm{int}(\svee)$
since $\lambda > 0$. Thus
\[
x^v \in \cJ_{\omega_X}((x^u)^\lambda) \iff v \in \lambda u + \lambda \mathrm{int}(\svee) \iff v - \lambda u \in \lambda \mathrm{int}(\svee).
\]
Using the definition of $\mathrm{int}(\svee)$
\[
v - \lambda u \in \lambda \mathrm{int}(\svee) \iff \langle v - \lambda u, n \rangle > 0 \quad \forall n \in \sigma \setminus \{0\}.
\]
Since $\sigma$ is generated by the $v_i$, this simplifies to
\[
\langle v, v_i \rangle > \lambda \langle u, v_i \rangle \quad \forall i = 1, \dots, s. \quad (\star)
\]

On the other hand,
\[
\left\lceil K_X - \lambda \cdot \Div(x^u) \right\rceil = \sum_{i=1}^s \left\lceil -(1 + \lambda \langle u, v_i \rangle) \right\rceil D_i = -\sum_{i=1}^s \lfloor 1 + \lambda \langle u, v_i \rangle \rfloor D_i,
\]
using the identity $\lceil -x \rceil = -\lfloor x \rfloor$ for $x \in \R$. A monomial $x^v$ is a section of $\cO_X(\lceil K_X - \lambda \Div(x^u) \rceil)$ if and only if

\[
\Div(x^v) + \left\lceil K_X - \lambda \cdot \Div(x^u) \right\rceil \geq 0,
\]
that is,
\[
\sum_{i=1}^s \langle v, v_i \rangle D_i - \sum_{i=1}^s \lfloor 1 + \lambda \langle u, v_i \rangle \rfloor D_i \geq 0.
\]
This holds if and only if for each $i$,

\[
\langle v, v_i \rangle \geq \lfloor 1 + \lambda \langle u, v_i \rangle \rfloor. \quad (\star\star)
\]

Now, noting for $b \in \Z$
\[
b > \lambda a \iff b \geq 1 + \lfloor \lambda a \rfloor.
\]

gives $(\star) \iff (\star\star)$. 
\end{proof}

\begin{rmk}
    Alternately, one can prove Lemma \ref{BlickleModule} following ideas from \cite[Theorem A.2.]{ST-Appendix}.
\end{rmk}

\section{Toric geometry}

\subsection{Setup}
Let $X \ra \Spec R$ be a log resolution of singularities of $(\Spec R, \fa)$. Fix a closed point $x \in X$. By Cohen's Structure Theorem, $\cOmpletion_{X,x} \cong k[[s_1, \dots, s_n]]$. Assume  $\fa \cdot \cOmpletion_{X, x} = (f)$ where $f = u \cdot s_1^{a_1} \cdots s_m^{a_m}$ with $m \leq n$, $u$ is a unit. Consider the deformation to the normal cone
\[
\sX \defeq \Spec_X \left( \bigoplus_{n \in \mathbb{Z}} \fa^n \cO_X \right) = \Spec_X \left( \cO_X[\fa t, t^{-1}] \right).
\]
Let $\resdiv \defeq V(t^{-1})$. 

\begin{lemma}\label{localToric}
Let $\xi \in \mathscr{X}$ be a closed point over $x \in X$. There exists an affine toric variety $\mathcal{X}_f = \Spec \ B = \Spec \left( k[x,y,s_1,\dots,s_n]/(xy - f) \right)$ and an isomorphism of complete local rings:
\[
\iota: \widehat{\mathcal{O}}_{\mathcal{X}_f,0}  \xrightarrow{\sim} \widehat{\mathcal{O}}_{\mathscr{X},\xi} 
\]
where, $\iota: x \mapsto f\cdot t$, $y \mapsto t^{-1}, \mathrm{ and ,} s_i \mapsto s_i$.
\end{lemma}

\begin{proof}
By Cohen's Structure Theorem, $\widehat{\mathcal{O}}_{X,x} \cong k[[s_1,\dots,s_n]]$ and $\mathfrak{a} \cdot \widehat{\mathcal{O}}_{X,x} = (f)$ where $f = u \cdot s_1^{a_1}\cdots s_m^{a_m}$, $u$ being a unit. 

The complete local ring at $\xi$ is
\[
\widehat{\mathcal{O}}_{\mathscr{X},\xi} \cong k[[f\cdot t,t^{-1},s_1,\dots,s_n]]
\]

For $\mathcal{X}_f = \Spec \left( k[x,y,s_1,\dots,s_n]/(xy - f) \right)$, the completion at $0 = V(x,y,s_1,\dots,s_n)$ is
\[
\widehat{\mathcal{O}}_{\mathcal{X}_f,0} \cong k[[x,y,s_1,\dots,s_n]]/(xy - f)
\]

The map $\iota: x \mapsto f\cdot t$, $y \mapsto t^{-1}, \ \mathrm{ and ,} \ s_i \mapsto s_i$ gives the required isomorphism.
\end{proof}

Consider the ring $ B = k[x, y, s_1, \dots, s_n] / (xy - u \cdot s_1^{a_1} \cdots s_m^{a_m}) $ where $u$ is a unit. Note, by replacing $x$ with $xu^{-1}$, we may assume that $f = s_1^{a_1} \cdots s_n^{a_m}$. Also, note, $B$ is an integral domain, normal, and the ideal $\mathfrak{m} = (x,y,s_1,\dots,s_n)$ is a maximal ideal. Set $ \cX = \Spec(B) $ and $ D_y = V(y) $. We prove $\cX$ is an affine toric variety (i.e, normal, with dense torus action).

\subsection{Affine Toric Model}\label{step:toric-model} 

Define $B \defeq k[x,y,s_1,\dots,s_n] / (xy - f)$ with $\mathfrak{m} \defeq (x,y,s_1,\dots,s_n)$. Let $\lattice \defeq \Z^{n+2}$ with basis $\basis{x}, \basis{y}, \basis{s_1}, \dots, \basis{s_n}$. The relation $(xy = f)$ corresponds to $v \defeq (1,1,-a_1,\dots,-a_m,0,\dots,0) \in \lattice$. Define $N \defeq \lattice / \Z v$ of rank $n+1$. The dual lattice is given by
\[
\duallattice \defeq \left\{ (b_x, b_y, b_1, \dots, b_n) \in \Z^{n+2} \mid b_x + b_y = \sum_{j=1}^m a_j b_j \right\}.
\]
%The cone $\cOne \subset N_\R$ is generated by $d_x \defeq q(\basis{x})$, $d_y \defeq q(\basis{y})$, $d_j \defeq q(\basis{s_j})$. The dual cone:
%\[
%\dualcone \defeq \left\{ u \in \duallattice_\R \mid u \geq 0 \text{ componentwise} \right\}.
%\]
Define $\phi: k[\dualcone \cap \duallattice] \ra B$ by $\chi^{(b_x, b_y, \mathbf{b})} \mapsto x^{b_x} y^{b_y} \mathbf{s}^{\mathbf{b}}$.

Thus, $\phi$ is an isomorphism and $B$ is affine toric. The completion $R \cong \widehat{B}_{\mathfrak{m}}$ is the completion of an affine toric variety at its torus-fixed point.

\subsection{Explicit Torus Action}\label{step:torus-action}
Define the torus $\torus \defeq (\Gm)^{n+1} = \Spec k[t_1^{\pm 1}, \dots, t_n^{\pm 1}, t_{n+1}^{\pm 1}]$. The action on $B$:
\begin{align*}
    (t_1,\dots,t_n,t_{n+1}) \action x & \defeq t_{n+1} x \\
    (t_1,\dots,t_n,t_{n+1}) \action y & \defeq \left( \prod_{j=1}^m t_j^{a_j} \right) t_{n+1}^{-1} y = t_y \cdot y\\
    (t_1,\dots,t_n,t_{n+1}) \action s_j & \defeq t_j s_j \quad (j=1,\dots,n)
\end{align*}
Preserves the relation:
\[
(\bft \action x)(\bft \action y) = \left( \prod_{j=1}^m t_j^{a_j} \right) xy = \left( \prod_{j=1}^m t_j^{a_j} \right) f = \prod_{j=1}^m (\bft \action s_j)^{a_j}.
\]
The point $p = (1,1,\dots,1)$ has trivial stabilizer, so $\torus \cdot p$ is dense. The divisor $\resdiv = V(y)$ is invariant since:
\[
\bft \action y = \left( \prod_{j=1}^m t_j^{a_j} \right) t_{n+1}^{-1} y.
\]

\subsection{Dense Torus Orbit}

Assume $ u = 1 $ by replacing $ x $ with $ x/u $. Then relation becomes $ (x/u)y = \prod s_i^{a_i} $. Set $ U = \cX \setminus V(x y s_1 \cdots s_n) $. Define morphisms

\[
\phi: T \to U, \quad \phi(\mathbf{t}) = \left( t_0 \prod_{j=1}^n t_j^{a_j},  t_0^{-1},  t_1, \dots, t_n \right)
\]

\[
\psi: U \to T, \quad \psi(x,y,s_1,\dots,s_n) = \left( y^{-1}, s_1, \dots, s_n \right)
\]

It is easy to verify that 
\begin{align*}
    \psi \circ \phi(\mathbf{t}) = \mathbf{t} \\
    \phi \circ \psi(\xi) = \xi
\end{align*}

Thus, $ \phi: T \to U $ is an isomorphism with inverse $ \psi $. The $ T $-action on $ U $ is transitive and free, so $ T $ is dense in $ X $.

\subsection{Normality of $ \cX $}

$ B $ is integral domain. To show normality, we verify Serre's conditions (R1) and (S2).

\textbf{(S2):} $ B $ is Cohen-Macaulay (a hypersurface), hence (S2).

\textbf{(R1):} Singular locus has codimension $ \geq 2 $. The Jacobian ideal
\[
J = \left( \frac{\partial f}{\partial x}, \frac{\partial f}{\partial y}, \frac{\partial f}{\partial s_1}, \dots, \frac{\partial f}{\partial s_n} \right) = \left( y, x, -a_1 s_1^{a_1-1} \prod_{j \neq 1} s_j^{a_j}, \dots, -a_n s_n^{a_n-1} \prod_{j \neq n} s_j^{a_j} \right)
\]
Singular locus $ V(J) \subseteq V(x,y) $. Thus, $ B $ is normal (see Lemma \ref{normality} for details).

\subsection{Torus-Invariant Prime Divisors}\label{torusinvdivisor}
First note that 
\begin{align*}
    V(x) = \bigcup_{i=1}^m V(x, s_i^{a_i}) \\
    V(y) = \bigcup_{i=1}^m V(y, s_i^{a_i}) \\
    V(s_i) = V(x, s_i) \bigcup V(y, s_i) \\
\end{align*}
\begin{proposition}
The torus-invariant prime divisors on $U_\sigma = \Spec B$ are:

\begin{align*}
D_{x,i} &:= V(x, s_i) \quad (1 \leq i \leq m), \\
D_{y,i} &:= V(y, s_i) \quad (1 \leq i \leq m), \\
D_i &:= V(s_i) \quad (m+1 \leq i \leq n).
\end{align*}

These correspond to the rays of the cone $\sigma$.
\end{proposition}
\begin{proof}

Each ideal is prime:

\begin{align*}
A/(x, s_i) &\cong k[y, s_1, \dots, \widehat{s_i}, \dots, s_n], \\
A/(y, s_i) &\cong k[x, s_1, \dots, \widehat{s_i}, \dots, s_n], \\
A/(s_i) &\cong k[x, y, s_1, \dots, \widehat{s_i}, \dots, s_n]/(xy - s_1^{a_1} \cdots s_m^{a_m}) \quad (i > m).
\end{align*}
These are integral domains of dimensions n.

The coordinate ring $A$ has monomial generators. The torus-action scales variables. So each is torus-invariant
\begin{align*}
    \mathbf{t} \cdot (y, s_i) &= (t_y y, t_i s_i) \subseteq (y, s_i) \\
    \mathbf{t} \cdot (x, s_i) &= \left( t_{n+1} x, t_i s_i \right) \subseteq (x, s_i)
\end{align*}
By the orbit-cone correspondence, invariant divisors correspond to rays. 
\begin{itemize}
\item $D_x$ corresponds to the ray $\rho_x = \mathbb{R}_{\geq 0} \cdot e_x$,
\item $D_y$ corresponds to $\rho_y = \mathbb{R}_{\geq 0} \cdot e_y$,
\item $D_i$ corresponds to $\rho_i = \mathbb{R}_{\geq 0} \cdot e_i$,
\end{itemize}
where $e_x, e_y, e_i$ are basis vectors in the lattice $N = \mathbb{Z}^{n+2}/\langle (1,1,-a_1,\dots,-a_m,0,\dots,0) \rangle \cong \mathbb{Z}^{n+1}$.
\end{proof}

\begin{lemma}\label{ToricCanonical}
The canonical divisor is
\[
K_{\mathcal{X}_f} = -\sum_{i=1}^m D_{x,i} - \sum_{i=1}^m D_{y,i} - \sum_{i=m+1}^n D_i
\]
\end{lemma}

\begin{rmk}\label{prop:global-action}
The torus can be made to act on $\sX$ via
\begin{align*}
\bft \cdot g &= g \quad \forall g \in \cO_X, \\
\bft \cdot (a t^k) &= a \cdot (t_y^{-k} t^k) \quad ,\forall a \in \fa^k,  \textit{as in Subsection \ref{step:torus-action}}\\
\bft \cdot t^{-1} &= t_y \cdot t^{-1}.
\end{align*}
This action restricts to $\phi^*$-conjugate action on $U$: For $u \in U$,
\[
\bft \star u = \phi^{-1}\left( \bft \cdot \phi(u) \right).
\]
\end{rmk}

\begin{lemma}\label{GlobalBlickle}
Let $\sX = \Spec_X (\bigoplus_{n \in \Z} \mathfrak{a}^n \cO_X)$ with $D = V(t^{-1})$. Then
\[
\mathcal{J}(\omega_{\sX}, \lambda D) \cong \cO_{\sX}(\lceil K_{\sX} - \lambda D \rceil)
\]
\end{lemma}

\begin{proof}
Fix a closed point $\xi \in \sX$, and apply Lemma \ref{localToric}. Define $P := \Spec$ $\mathcal{O}_{\mathscr{X},\xi}$ and $Q= \Spec$ $\mathcal{O}_{\mathcal{X}_f,0}$. Then, by Lemma \ref{BlickleModule} we have 

\[
\mathcal{J}(\omega_{Q}, D_y^\lambda)  \cong \cO_{Q}(\lceil K_{Q} - \lambda D_y \rceil) 
\]

As multiplier module commutes with localization, we get,

\[
\mathcal{J}(\omega_{Q}, D_y^\lambda) = \mathcal{J}(\omega_{P}, D^\lambda) \cong \cO_{P}(\lceil K_{P} - \lambda D \rceil) 
\]

As both $\mathcal{J}(\omega_{\sX}, D^\lambda)$ and $\cO_{\sX}(\lceil K_{\sX} - \lambda D \rceil)$  are coherent subsheaves of $K_\sX$ matching at every closed point $\xi$, we get, 
\[
\mathcal{J}(\omega_{\sX}, D^\lambda)  \cong \cO_{\sX}(\lceil K_{\sX} - \lambda D \rceil) 
\]

\end{proof}

\section{Main Theorems}

In this section, we prove our main result. Before doing that, we need the following lemma.

\begin{lemma}\label{ProjectiveToroidal}
 Let $\cT := \bigoplus\limits_{n \in \Z} \fa^nt^n$ be the Extended Rees algebra of $R$ with respect to $\fa$, where, $\fa^n := R$ for $n \leq 0$. Let $\psi:  X=\Proj \ R[J]) \longrightarrow \Spec R$ be a log resolution of $(\Spec R, \fa).$ Then,
 $$\sX = \Spec_X\left(\bigoplus\limits_{n\in\Z}\fa^n\cO_X t^n\right) \simeq \Proj \cT[J s],$$
    where $\cT[J s]= \cT\oplus(J s)\oplus(J^2s^2)\oplus\cdots$ and $s$ is a dummy variable of degree 1. In particular, $$\beta : \sX = \Spec_X\left(\bigoplus\limits_{n\in\Z}\fa^n\cO_X t^n\right) \simeq \Proj \cT[J s] \longrightarrow \Spec \ \cT $$ is projective and birational morphism.
\end{lemma}

\begin{proof}
We have the following commutative diagram:

    \[\begin{tikzcd}
	{\sX} & {X=\Proj R[\fa s]} \\
	{\Spec \cT} & {\Spec R}
	\arrow["{\tilde{\pi}}", from=1-1, to=1-2]
	\arrow[from=1-1, to=2-1]
	\arrow["\phi", from=1-2, to=2-2]
	\arrow[from=2-1, to=2-2]
   \end{tikzcd}\]
Pick a basic affine open $\Spec R \Jf \simeq D_+(fs)\subset X$, where $f\in\fa$. Then the affine coordinate ring of $\tilde{\pi}\inv(D_+(fs))$ is 

    $$
    \begin{aligned}
        \bigoplus_{n\ge0}R\Jf \fa^nt^n\bigoplus_{n<0}R\Jf t^n & =R\Jf[\fa t,t\inv]\\
        &=\cT\Jf\\
        &=(\cT[\Jf s]_{fs})_0.
    \end{aligned}
    $$

    On the other hand, this is precisely the affine coordinate ring of the basic open subset $D_+(fs)$ of $\Proj \cT[\fa s]$. As the above glues well, we get the global isomorphism.
\end{proof}

\begin{thm}
\label{theorem:multipliermodule-decomp}
Let $(R,\mathfrak{m})$ be a normal local ring of dimension at least $2$, finite type over a field of characteristic $0$. Let $\fa$ be an $R$-ideal with $\mathrm{ht}(\fa)>0$.

Write $\cS = \bigoplus\limits_{n \ge0} \fa^nt^n$ for the Rees algebra of $\fa$ and 
$\cT := \bigoplus\limits_{n \in \Z} \fa^nt^n$ for the extended Rees algebra of $R$ with respect to $\fa$, where, $\fa^n := R$ for $n \leq 0$. Let $\lambda \geq 0$ be any real number.

Then 
\begin{enumerate}
    \item $\cJ(\omega_{\cS},(\fa\cdot \cS)^\lambda) = \bigoplus\limits_{n\geq 0} \cJ(\omega_R, \fa^{n+1+\lambda})t^{n+1}.$\label{ReesMultiplier}

    \item $\cJ(\omega_{\cT}, (t^{-1})^\lambda) = \bigoplus\limits_{k \in\Z} \cJ (\omega_R, \fa^{k+\lambda})t^k.$ \footnote{With the convention that when $k+  \lambda  \leq 0$, we define $\fa^{k+  \lambda } = R$ }\label{ExtReesMultiplier}
\end{enumerate}

In particular, $[\cJ(\omega_{\cT}, (t^{-1})^\lambda)]_0 =  \cJ (\omega_R, \fa^{\lambda})$ and $[\cJ(\omega_{\cT}, (t^{-1})^\lambda)]_{\geq 1} = \cJ(\omega_{\cS},(\fa\cdot \cS)^\lambda)$ .

\end{thm}

\begin{proof}

For (\ref{ReesMultiplier}), we will follow \cite[Remark 2.8]{Hyry-Villamayor} closely. We can find an ideal $J \subset R$ such that $\psi : X=\Proj \ R[J] \longrightarrow (\Spec R, \fa)$ is a log resolution of singularities with $\fa \cdot \cO_X = \cO_X (-E)$ for some effective, Cartier, divisor $E$. Set $Y=\Proj \ S [J \cdot \cS]$. By looking the affine pieces (similar to Lemma \ref{ProjectiveToroidal}) which cover $Y$, one easily checks that $Y=\Spec_X\bigoplus\limits_{n \geq 0} \fa^{n} \cO_{X}t^n$. Since $\fa \mathcal{O}_{X}$ is invertible, this means that $Y=\mathbb{V}\left(\fa \mathcal{O}_{X}\right)$ is the total space of the tautological line bundle $\cO_X(-E)$, see \cite[Subsection 6.2.1]{HyrySmith}. Since $X$ is regular, so is $Y$. The canonical projection $\phi: Y \longrightarrow \Spec \cS$ is thus a desingularization of $\operatorname{Spec} \cS$. Let $\pi: Y \longrightarrow X$ be the canonical morphism. Then $\pi_{*} \mathcal{O}_{Y}=\bigoplus_{n \geq 0} \fa^{n} \mathcal{O}_{X}t^n$. We have the following diagram

\[\begin{tikzcd}
        {Y=\Spec_X(\bigoplus\limits_{n\ge0} \fa^n\cO_X t^n)} & {\Spec \cS} \\
	X & {\Spec R}
	\arrow["\phi", from=1-1, to=1-2]
	\arrow["\pi"', from=1-1, to=2-1]
	\arrow["{\pi'}", from=1-2, to=2-2]
	\arrow["\psi"', from=2-1, to=2-2]
\end{tikzcd}\]
Note, $\pi:Y\to X$ is affine and smooth of relative dimension 1. By Subsection \ref{GrothNatural}, we also know,
$$\omega_Y=\pi^*\omega_X\otimes\Omega_{Y/X}
=\pi^*\omega_X\otimes(\mathfrak{a}t)\cdot\mathcal{O}_Y.$$
Thus,
$$\pi_*\omega_Y= \w_X \otimes \left(\bigoplus_{n \geq 0} \fa^{n+1}\cdot t^{n+1}\right) = \bigoplus_{n \geq 0}(\fa^{n+1} \cdot \w_X)t^{n+1}$$
As $\pi$ is affine,
$$H^i(Y,\omega_Y)=H^i(X,\pi_*\omega_Y)   =\bigoplus\limits_{n\ge 0} H^i(X,\mathfrak{a}^{n+1}\omega_X)t^{n+1} = \bigoplus\limits_{n\ge 0} H^i(X,\omega_X(-(n+1)E))t^{n+1}, \ \forall i \geq 0 $$
Now, 

\begin{align*}
\cJ(\w_\cS, (\fa \cdot \cS)^\lambda) &= \phi_*\cO_Y(\lceil K_Y - \la \pi^*E     \rceil) \\
&= H^0(Y, \w_Y(\lceil -\la \pi^*E \rceil)) \\
&= \bigoplus\limits_{n\ge 0} H^0(X,\omega_X(\lceil -(n+1+\la)E \rceil))t^{n+1} \\
&= \bigoplus\limits_{n\ge 0} \cJ(\w_R , \fa^{n+1+ \la})t^{n+1}
\end{align*}
This proves part (\ref{ReesMultiplier}).\\
To prove the statement about the extended Rees algebra in part (\ref{ExtReesMultiplier}), we consider the diagram

    \[\begin{tikzcd}
	{\sX =\Spec_X\left(\bigoplus\limits_{k\in\Z}\fa^k\cO_Xt^k\right)}       & X \\
	{\Spec T} & {\Spec R}
	\arrow["{\tilde{\pi}}", from=1-1, to=1-2]
	\arrow["\beta"', from=1-1, to=2-1]
	\arrow["\psi", from=1-2, to=2-2]
	\arrow[from=2-1, to=2-2]
    \end{tikzcd}\]
Pick a dominating (log) resolution
    $$\cZ\xrightarrow{\theta}\sX\xrightarrow[]{\beta}\Spec \cT.$$
Now from Lemma \ref{GlobalBlickle}, we know that

    $$\cJ(\omega_\sX,(t\inv)^\lambda)=\cO_\sX(\lceil K_\sX-\lambda\Div(t\inv)\rceil).$$
So, noting that $\beta$ is projective, and, birational, from Lemma \ref{ProjectiveToroidal}, we have,

    $$
    \begin{aligned}
        \cJ(\omega_\cT,(t\inv)^\lambda)&=\beta_*\theta_*\cO_\cZ(\lceil K_\cZ-\lambda(\beta\circ\theta)^*(\Div(t\inv))\rceil )\\
        &= \beta_* \cJ (\w_\sX, \beta^*(\Div(t^{-\la}))) \\
        &=\beta_*\cO_\sX(\lceil K_\sX-\lambda\beta^*\Div(t\inv)\rceil)\\
        \end{aligned}
    $$
i.e, we can use $\beta$ to compute multiplier module $\cJ(\omega_\cT,(t\inv)^\lambda).$

Recall from Lemma \ref{localToric}, for any closed point \(\xi \in \mathscr{X}\) over \(x \in X\), there exists an isomorphism of complete local rings
\[
\iota: \widehat{\mathcal{O}}_{\mathscr{X},\xi} \xrightarrow{\sim} \widehat{\mathcal{O}}_{\mathcal{X}_f,0}, \quad \mathcal{X}_f = \operatorname{Spec} k[x,y,s_1,\ldots,s_n]/(xy-f), \ \text{for} \ f = s_1^{a_1}\cdot ... \cdot s_m^{a_m}, \ m \leq n.
\]
such that $D = V(t^{-1})$ corresponds to $D_y = V(y)$.
The canonical divisor is given in Lemma \ref{ToricCanonical} by
\[
K_{\mathcal{X}_f} = -\sum_{i=1}^m D_{x,i} - \sum_{i=1}^m D_{y,i} - \sum_{i=m+1}^n D_i
\]
Recall, from Subsection \ref{torusinvdivisor},
\begin{align*}
    V(y) = \sum_{i=1}^m a_iD_{y,i} \ ,  \
     V(x) = \sum_{i=1}^m a_iD_{x,i} 
\end{align*}
we get,

\begin{align}
\lceil K_{\mathcal{X}_f} - \lambda D_y\rceil &= -\sum_{i=1}^m D_{x,i} - \sum_{i=1}^m \lceil 1+\lambda  a_i \rceil D_{y,i} - \sum_{i=m+1}^n D_i \label{Canonicalmultiplier}
\end{align}
Also, \(\iota\) identifies the multiplier ideals
\[
\widehat{\mathcal{J}(\omega_{\mathscr{X}}, D^\lambda)}_\xi = \widehat{\mathcal{O}}_{\mathscr{X}, \xi} \otimes_{\mathcal{O}_{\mathscr{X}, \xi}} \mathcal{J}(\omega_{\mathscr{X}}, D^\lambda)_{\xi} \cong \widehat{\mathcal{O}}_{\mathcal{X}_f, 0} \otimes_{\mathcal{O}_{\mathcal{X}_f, 0}} \mathcal{O}_{\mathcal{X}_f}([K_{\mathcal{X}_f} - \lambda D_y])_0. 
\]\label{eq:MultToricExtRees}
What we want to do next is to show that, upto completion,  $\mathcal{O}_{\mathcal{X}_f,0}([K_{\mathcal{X}_f,}-\lambda D_y])$, under $x \longrightarrow ft$, $y \longrightarrow t^{-1}$ and $s_i \longrightarrow s_i$ decomposes as 
\[
\bigoplus_{k< -\lambda}\cO_{X, x} \left( K_X \right) t^k \bigoplus_{k \geq - \lambda} \cO_{X, x} \left( \lceil K_X -  (k + \lambda) E \rceil \right) t^k
\]
From \eqref{Canonicalmultiplier}, we see that a monomial $x^a y^{b} s_1^{c_1}\cdots s_n^{c_n}$ is a section of $\mathcal{O}_{\mathcal{X}_f,0}([K_{\mathcal{X}_f,}-\lambda D_y])$ iff:
\begin{align}
aa_i + c_i &\geq 1 \quad \forall 1 \leq i \leq m \label{eq:cond1} \\
ba_i + c_i &\geq \lceil 1+\lambda  a_i \rceil \quad \forall 1 \leq i \leq m \label{eq:ci-lower} \\
c_i &\geq 1 \quad \forall m+1 \leq i \leq n \label{eq:ci-upper}
\end{align}
Under $\iota^{-1}$, the monomial $x^a y^{b} \mathbf{s}^{\mathbf{c}}$ maps to
\[
g_{a,c}\cdot t^{a-b} \in \widehat{\mathcal{O}}_{\mathscr{X},\xi}, \quad \text{where} \quad g_{a,c} := \prod_{i=1}^m s_i^{c_i + a a_i} \prod_{i=m+1}^n s_i^{c_i}
\]
Hence, denoting $E = V(f) \subset X$, we see that $g_{a,c}t^{a-b} \in \mathcal{O}_{X,x}(\lceil K_X - (a-b+\lambda)E \rceil)t^{a-b}$ iff:
\begin{align}
c_i + a a_i \geq \lceil  (a-b + \lambda) a_i  + 1\rceil \quad \forall i \leq m \label{multipliercondition} \\
c_i \geq 1 \quad \forall i > m \label{trivialsection}
\end{align}
when $(a-b + \lambda) \geq 0$, in which case 
$c_i + a a_i \geq 1 \quad \forall i \leq m $ \\
When $(a-b + \lambda) < 0$, as we want things to be inside $H^0(X, \w_X)$, we get $c_i + a a_i \geq 1 \quad \forall i \leq m $. Note, this implies \eqref{multipliercondition}.\\
Finally, and most importantly, note that the first set of inequalities is exactly the same as the second set of inequalities.\\
As all sheaves involved are coherent and agree locally upto completion, hence they agree globally. We get, 
\begin{align*}
\tilde{\pi}_* \cO_{\mathscr{X}} \left( \lceil K_{\mathscr{X}} - \lambda D \rceil \right) 
&= \bigoplus_{k< -\lambda}\cO_{X} \left( K_X \right) t^k \bigoplus_{k \geq - \lambda} \cO_{X} \left( \lceil K_X -  (k + \lambda) E \rceil \right) t^k 
\end{align*}
where we substitute $k$ for $a-b$.
Hence, we get
 $$
    \begin{aligned}
        \cJ(\omega_\cT,(t\inv)^\lambda)&=\psi_*(\bigoplus_{k< -\lambda}\cO_{X} \left( K_X \right) t^k \bigoplus_{k \geq - \lambda} \cO_{X} \left( \lceil K_X -  (k + \lambda) E \rceil \right) t^k) \\
        &= \bigoplus_{k< -\lambda} H^0(X, \w_X)t^k \bigoplus_{k \geq - \lambda} H^0(X, \cO_{X} \left( \lceil K_X -  (k + \lambda) E \rceil \right) t^k \\
        &= \bigoplus_{k \in \Z} \cJ( \w_R, \fa^{k+\la})t^k 
        \end{aligned}
    $$
    with the convention that when $k+  \lambda \leq 0$, we define $\fa^{n+  \lambda } := R$. This proves Part (\ref{ExtReesMultiplier}) and we are done.
\end{proof}

\begin{rmk}\label{Multiplierideal}
    Assuming Gorensteinness, one gets similar decompositions for multiplier ideals, just by using Lemma \ref{PropertiesOfMultiplier}.
\end{rmk}

\begin{rmk}
    A similar decomposition for the Rees Algebra $\cS$ is obtained by Hyry under the assumption that $R$ is regular and $\Proj \cS$ has rational singularities, see \cite[Proposition 3.1]{HyryBlowUp}. Kotal and Kummini also proved similar decomposition for the Rees Algebra $\cS$, under the assumption that $R$ and $\cS$ are Cohen-Macaulay, see \cite[Theorem 1.3]{KotalKummini}.
\end{rmk}

\begin{rmk}
    A similar decomposition for the multiplier ideals of the extended Rees algebra $\cT$ is obtained by Budur, Mustaţă and Saito as well, under smoothness assumptions, using V-filtrations, see \cite[Theorem 1 along with (1.3.1)]{Budur-Mustata-SaitoBS}.\footnote{I am grateful to Bradley Dirks for pointing that out.} After finishing writing this article, I learnt from Bradley Dirks that he also proved a similar decomposition for the multiplier modules of the extended Rees algebra $\cT$ using V-filtrations in his upcoming paper \cite{Brad}.
\end{rmk}

\begin{rmk}\label{multipliermodulefortriples}
    Essentially following the same argument and the fact that for a principal divisor $D = (f)$, $\cJ(\Spec \ R, \Delta+D, \mathfrak{a}^t)=f.\cJ(\Spec \ R, \Delta, \mathfrak{a}^t)$, one can get similar decompositions for the test module of triples, just by keeping track of grading. We don't include it here for the sake of brevity.
\end{rmk}

\begin{rmk}
    With Hunter Simper, the author proved these decompositions for test modules in positive characteristics (see \cite[Theorems 4.1 and 4.3]{Ajit-Simper-Char-p}, using different techniques. By reduction-mod-p arguments, we get the above decompositions for multiplier modules!
\end{rmk}

\section{Applications}
Here we collect some quick applications directly following from the main Theorem \ref{theorem:multipliermodule-decomp}.

\subsection{Equivalence of rationality of Rees and extended Rees Algebras}
Throughout this section we maintain the notation of the previous section, that is: We let $(R,\fm)$ be an $F$-finite, reduced ring of dimension $d\geq 2$ and $\mathfrak{a}\subseteq R$ an ideal with positive height. We set $\rS=R[\mathfrak{a}t]$ to be the Rees algebra of $\mathfrak{a}$ and $\cT=R[\mathfrak{a}t,t^{-1}]$ the extended Rees algebra of $\mathfrak{a}$.

\begin{thm}(see also \cite{HWY-F-Regular})
    \label{Th-equivofFrationality}
    Under above notation, we have the following implications:
    \begin{enumerate}
        \item $R$ and $\cS$ are both rational implies $\cT$ is rational.
        \item $\cT$ is rational implies $R$ and $\cS$ are rational.
    \end{enumerate}
\end{thm}

\begin{proof}
    First we prove the equivalence about Cohen-Macaulayness. Note from \cite[Proposition 1.1]{Huneke-82}  it follows that $R$ and $\cS$ are Cohen-Macaulay implies $G= \cT/(t^{-1})$ is Cohen-Macaulay and hence $\cT$ is Cohen-Macaulay as well. To prove the other direction, note the following:
    \begin{itemize}
        \item $\cT$ is Cohen-Macaulay if and only if $G$ is Cohen-Macaulay.
        \item $\cT$ is rational implies $ T_{t^{-1}} = R[t, t^{-1}]$ is rational, which in turn implies $R$ is rational and hence, $R$ is pseudo-rational by \cite[Theorem 5.10]{KollarMori98}.
        \item $\cT$ being Cohen-Macaulay and $R$ being pseudo-rational implies, from  \cite[Theorem (5)]{Lipman1994CM} that $\cS$ has to be Cohen-Macaulay as well.
    \end{itemize}

    Hence, we need only compare the multiplier modules of $R$, $\cS$, and $\cT$. With this in mind, both (1) and (2) are immediate from Theorem \ref{theorem:multipliermodule-decomp} by putting $\lambda = 0$.
\end{proof}

\begin{rmk}\label{deformation}
    If $G = \cT/(t^{-1})$ is rational, then we get $\cT$ is rational and hence $\cS$ is rational. Similar results were known for Cohen-Macaulay and Gorenstein property by extensive works of Goto, Huneke, Ikeda, Lipman, Shimoda, Vi\'et etc, see \cite{GotoShimodaCM}, \cite{iai2024characterizationsgorensteinreesalgebras}, \cite{Goto-Nishida-Book}, \cite{Ikeda-Gor}, \cite{VietCM}, \cite{WhenCM}, \cite{WhenGor} \cite{Huneke-82} and, \cite{Lipman1994CM} for example.
\end{rmk}

\begin{App}
    It directly follows from Remark \ref{Multiplierideal} that if $\lambda$ is an jumping number for $(R, \fa)$ then so is $\lambda + n$ for infinitely many $n \geq 0$, see \cite[Proposition 1.12]{ELSV}.

\end{App}

There are a few interesting future directions: Core computations as in \cite{HyrySmith} and \cite{Hyry-Smith-Core}, Bernstein-Sato polynomials \footnote{After finishing writing this article,  Bradley Dirks informed me that he proved related results in his upcoming paper \cite{Brad}.} (\cite{Budur-Mustata-SaitoBS}), developing similar decompositions in mixed characteristics \cite[Proposition 5.5]{7author-1}, developing a multi-Rees algebra analog to prove various properties of multiplier ideals, etc.

\bibliographystyle{alpha}
\bibliography{ref}

\end{document}